\documentclass[12pt]{article}
\usepackage{e-jc}
\usepackage{amsmath, amsthm, amscd, amsfonts, amssymb, graphicx, color, enumerate, xcolor,  mathrsfs, latexsym}

\usepackage{graphics}
\usepackage{color,soul}

\theoremstyle{plain}
\newtheorem{theorem}{Theorem}

\newtheorem{corollary}[theorem]{Corollary}

\theoremstyle{definition}

\newtheorem{example}[theorem]{Example}
\newtheorem{conjecture}[theorem]{Conjecture}

\theoremstyle{remark}
\newtheorem{remark}[theorem]{Remark}

\title{\bf Bounds for Distinguishing Invariants\\ of Infinite Graphs}
\author{Wilfried Imrich \\
\small Montanuniversit\"at Leoben, A-8700 Leoben, Austria \\
\small\tt imrich@unileoben.ac.at\\
\and 
Rafa{\l} Kalinowski \qquad Monika Pil\'sniak\thanks{Rafa{\l} Kalinowski and Monika Pil\'sniak were partially supported by the Polish Ministry of Science and Higher Education.}\\
\small AGH University, Department of Discrete Mathematics, 30-059 Krakow, Poland\\
\small\tt {kalinows,pilsniak}@agh.edu.pl
\and Mohammad Hadi Shekarriz\footnote{Corresponding author.}\\
\small Ferdowsi University of Mashhad, P.O. Box 1159, Mashhad 91775, Iran.\\
\small\tt mh.shekarriz@mail.um.ac.ir\\
}

\date{\dateline{Aug 6, 2016}{Jun 30, 2017}}

\begin{document}

\maketitle

\begin{abstract}  We consider infinite graphs. The distinguishing number $D(G)$ of a graph $G$ is the minimum number of colours in a vertex colouring of $G$ that is preserved only by the trivial automorphism. An analogous invariant for edge colourings is called the distinguishing index, denoted by $D'(G)$. We prove that $D'(G)\leq D(G)+1$.
For proper colourings, we study relevant invariants called the distinguishing chromatic number $\chi_D(G)$, and the distinguishing chromatic index $\chi'_D(G)$, for vertex and edge colourings, respectively. We show that $\chi_D(G)\leq 2\Delta(G)-1$ for graphs with a finite maximum degree $\Delta(G)$, and we obtain substantially lower bounds for some classes of graphs with infinite motion. We also show that $\chi'_D(G)\leq \chi'(G)+1$, where $\chi'(G)$ is the chromatic index of $G$, and we prove a similar result $\chi''_D(G)\leq \chi''(G)+1$ for  proper total colourings.
A number of conjectures are formulated.

\end{abstract}

\section{Introduction}
In this paper we intend to compare distinguishing invariants for infinite graphs and, when it is possible, to find bounds for them. To read this paper, some definitions from infinite graph theory are needed; a graph is called an infinite graph if it has infinitely many vertices. Furthermore, an infinite graph is called locally finite if all of its vertex degrees are finite. A \emph{ray} is an infinite graph with vertex set $V = \{ v_{1}, v_{2}, \ldots \}$ and edge set $\{ v_{1}v_{2}, v_{2}v_{3}, \ldots \}$ while a \emph{double ray} is an infinite graph with vertex set $V = \{\ldots, v_{-2}, v_{-1}, v_{0}, v_{1}, v_{2}, \ldots \}$ and edge set $\{\ldots, v_{-2}v_{-1}, v_{-1}v_{0},v_{0}v_{1}, v_{1}v_{2}, \ldots \}$. A \emph{rayless} graph is a graph that does not contain a ray as its subgraph. For a general reference about finite and infinite graph theory see \cite{Diestel}.

The study of symmetry breaking in graphs has its origin in \cite{Albertson}, where Albertson and Collins introduced the \emph{distinguishing number} $D(G)$ of a graph $G$ as the minimum number of colours required to colour the vertices of $G$ such that this colouring is only preserved by the trivial automorphism.  Such a colouring is called \emph{distinguishing}. For a connected finite graph $G$, it was proved in \cite{Collins} and \cite{klavzar} that $D(G) \leq \Delta + 1$, where $\Delta$ is the largest degree of $G$. Equality holds if and only if $G$ is a complete graph $K_{\Delta+1}$, a balanced complete bipartite graph $K_{\Delta,\Delta}$, or $C_5$.

In 2007, Imrich, Klav\v{z}ar and Trofimov \cite{Wilfried} considered the distinguishing number for infinite graphs. They showed that for an infinite connected graph $G$ we have $D(G) \leq \mathfrak{n}$, where $\mathfrak{n}$ is a cardinal number such that the degree of any vertex of $G$ is not greater than $\mathfrak{n}$. In 2011, one important conjecture about the distinguishing number of infinite graphs was made, based on the definition of the \emph{motion of a graph $G$}, a concept already introduced in \cite{Russell}. The motion of an automorphism is defined as the number of vertices which are moved by it, and the motion of a graph $G$, denoted by $m(G)$, is the minimum of motions of non-identity automorphisms of $G$. Tucker posed in \cite{Tucker} the Infinite Motion Conjecture that every locally finite infinite graph with infinite motion has distinguishing number at most 2.

The concept of symmetry breaking became subject of numerous variations after its appearance. In 2006, Collins and Trenk \cite{Collins} mixed the concept of distinguishing colourings with proper vertex colourings to introduce \emph{the distinguishing chromatic number $\chi_{D}(G)$} of a graph $G$. It is defined as the minimum number of colours required to properly colour the vertices of $G$ such that this colouring is only preserved by the trivial automorphism. They also showed that for a finite connected graph $G$, we have $\chi_D (G) \leq 2\Delta(G)$ and that equality holds only if $G$ is isomorphic to $K_{\Delta,\Delta}$ or $C_6$. Moreover, Collins and Trenk \cite{Collins} conjectured that a graph $G$ that is different from $K_{\Delta,\Delta}$ or $C_6$, has distinguishing chromatic number less than or equal to $2\Delta(G)-1$ and that equality holds if and only if $G$ is isomorphic to $K_{\Delta, \Delta-1}$ (they actually missed to mention that the distinguishing chromatic number of $K_{\Delta,\Delta-1}$ is $2\Delta-1$).

An analogous index for an edge colouring, i.e. the \emph{distinguishing index} $D'(G)$, has been introduced by Kalinowski and Pil\'sniak in \cite{Kalinowski} as the minimum number of colours needed to colour edges of a graph $G$ such that no non-trivial automorphism preserves it. Moreover, they showed that $D'(G) \leq \Delta(G)$ for finite connected graph $G$ unless $G$ is isomorphic to $C_3$, $C_4$ or $C_5$. They also introduced the \emph{distinguishing chromatic index $\chi'_D (G)$} as the minimum number of colours needed to properly colour the edges of $G$ such that this edge colouring is only preserved by the trivial automorphism. Furthermore, they showed that $\chi'_D (G)\leq \Delta(G) +1$ except for $C_4$, $K_4$, $C_6$ and $K_{3,3}$. Therefore, except for these four small graphs, we have $\chi'_D (G)\leq \chi'(G) +1$, a result which is also true for infinite graphs, as we show in Section 4.

In 2015, Broere and Pil\'sniak \cite{Izak} considered $D'(G)$ for infinite graphs. They proved that for every connected infinite graph we have $D'(G) \leq \Delta$, where $\Delta$ is a cardinal number such that degree of any vertex is at most $\Delta$. They also stated an \emph{infinite edge-motion conjecture} (analogous to the Infinite Motion Conjecture of Tucker) that was very recently confirmed by Lehner~\cite{Lehner}.

Kalinowski and Pil\'sniak \cite{Kalinowski} also compared $D(G)$ and $D'(G)$ when $G$ is a finite graph. They proved that if a connected finite graph $G$ has at least three vertices then $D'(G) \leq D(G) +1$. We extend this result to infinite graphs in Section 2.

The most recent generalization of the distinguishing number was made by Kalinowski, Pil\'sniak and Wo\'zniak in \cite{Rafal}. They introduced \emph{the total distinguishing number $D''(G)$} as the minimum number of colours needed to colour vertices and edges of a graph $G$ such that this colouring is only preserved by the trivial automorphism. It is clear by definition that $D''(G) \leq \mbox{max}\{D(G), D'(G)\}$. It was proved in  \cite{Rafal} that $D''(G)\leq \lceil\sqrt{\Delta(G)}\rceil$ for a connected graph $G$ with at least three vertices. Moreover, they introduced \emph{the total distinguishing chromatic number $\chi''_{D} (G)$} as the minimum number of colours needed to properly colour vertices and edges of $G$ in the way that only the trivial automorphism preserves it. The upper bound $\chi''_D (G) \leq \chi''(G)+1$ was also shown to hold for all finite connected graphs, and in Section 4 we show that it is also true for infinite graphs.

We start our discussion in Section 2 by comparing the distinguishing number and the distinguishing index of infinite graphs. In Section 3 we show that every infinite graph of bounded degree has distinguishing chromatic number at most $2\Delta -1$. We also derive some results for the distinguishing chromatic number of some graphs with infinite motion. And finally, the bounds for the distinguishing chromatic index and the total distinguishing chromatic number are proved in Section 4.

\section{Distinguishing Number and Distinguishing Index}

In this section we compare the distinguishing number and the distinguishing index of arbitrary infinite graphs. 

We extend Theorem 11 of Kalinowski and Pil\'sniak \cite{Kalinowski} to infinite graphs. Namely, we prove the following.

\begin{theorem}
Let $G$ be a connected infinite graph. Then $$D'(G) \leq D(G) +1.$$
\end{theorem}
\begin{proof}
Let $\hat{c}:V(G) \longrightarrow X$ be a distinguishing vertex colouring of $G$. We define a distinguishing edge colouring $c:E(G) \longrightarrow X\cup \{ p \}$ where $p \notin X$.

\begin{itemize}
\item[Case 1.] $G$ is a rayless tree. Then $G$ has either a central vertex or a central edge fixed by every automorphism by Theorem 2.5 of \cite{Polat}. Therefore, like in the proof of Theorem 9 in \cite{Kalinowski} we consider two subcases.
\begin{itemize}
\item[Subcase 1.1.] $G$ has a central vertex $v_0$. If $xy$ is an edge of $G$ such that $d(x,v_{0})=d(y,v_{0})+1$ then we colour it as $c(xy)=\hat{c}(x)$. If $\varphi$ is an automorphism of $G$ which preserves $c$, it must also fix $v_0$, and we have $\hat{c}(\varphi(x))=c(\varphi(x)\varphi(y)) =c(xy) =\hat{c}(x)$ for every vertex $x$ and $y$ with $y \in N(x)$ and $d(x,v_{0})=d(y,v_{0})+1$. This shows that $\varphi$ is the trivial automorphism because $\hat{c}$ is a distinguishing colouring.

\item[Subcase 1.2.] $G$ has a central edge $e_0 = a_{1}a_2$. Colour each edge $xy$ of $G-e_0$ with $\hat{c}(x)$ if the distance from $x$ to $e_0$ is greater than the distance from $y$ to $e_0$, else colour it with $\hat{c}(y)$. Colour the edge $e_0$ arbitrarily. If $\varphi$ is a non-trivial automorphism of $G$ preserving the colouring $c$, then there exist two edges $x_1 y_1$ and $x_2 y_2$ with the same colour such that $\varphi (x_{1})\varphi(y_{1}) = x_2 y_2$. Because $\varphi$ fixes the edge $e_0$, the distances from $e_0$ to $x_1$ and $x_2$ are equal, and $\hat{c}(x_1 ) = \hat{c}( x_2 )$ . Both edges $x_1 y_1$ and $x_2 y_2$ cannot belong to the same component of $G-e_0$, because then, by the definition of c, the automorphism $\varphi$ must also preserve $\hat{c}$, a contradiction. Therefore, $x_1 y_1$ and $x_2 y_2$ belong to different components of $G-e_0$ and hence, every edge-colour-preserving automorphism must exchange endpoints of $e_0$. If we colour one of these edges, say $x_1 y_1$, by an extra colour $p$, we break all these automorphisms.
\end{itemize}

Consequently, in this case $D'(G) \leq D(G) +1$ (and $D'(G) \leq D(G)$ when $G$ has a central vertex).

\item[Case 2.] $G$ is a tree with some rays. Choose an arbitrary vertex $v$ to be the root and make the tree rooted. Colour an arbitrary ray starting from $v$ with the new colour $p$ and colour other edges $xy$ by the colour $\hat{c}(x)$ if $d(x,v) = d(y,v) +1$. This colouring is distinguishing because the ray is asymmetric, so $v$ is fixed by any automorphism, and $c$ is also distinguishing, since $\hat{c}$ is.

\item[Case 3.] $G$ is not a tree. Then by the same arguments as in the proof of Theorem 11 of \cite{Kalinowski} we construct a suitable edge colouring $c$ with  $D(G) +1$ colours. Here we give only an outline of the reasoning. Assuming $D(G)\geq 2$, let $\hat{c}$ be a distinguishing vertex colouring of $G$ with colours $1,\ldots,D(G)$. We choose a shortest cycle $C$ and stabilise it by colouring two adjacent edges with $1$ and $2$, and all other ones with $0$. We will not use colour $0$ any more. We next colour every edge $xy$ with $c(xy)=\hat{c}(x)$ if the distance from $x$ to the cycle $C$ is one more than the distance from $y$ to  $C$.  Finally, we colour every edge $xy$ such that $x$ and $y$ are at the same distance from the cycle $C$ arbitrarily. The edge colouring $c$ is only preserved by the identity as $\hat{c}$ is so.
\end{itemize}
\end{proof}

\begin{remark}
By the above theorem,  $D'(G)\leq D(G)$, if $D(G)$ is infinite. Interestingly, there exist infinite graphs for which $D'(G)$ is indeed smaller than $D(G)$.
As an example, following~\cite{lm} we define the {\it graph on the rationals} as a graph $Q$ with the vertex set $V(Q)=\mathbb{Q}\times\{0,1\}$, and an edge  between $(a,1)$ and $(b,0)$ is introduced whenever $a<b$. Broere and Pil\'sniak in \cite{Izak} noticed that $D'(Q)=2$ while $D(Q)$ is infinite, which was shown by Lehner and M\"{o}ller in~\cite{lm}.
So $D(G)$ and $D'(G)$ can have an arbitrary large difference for infinite graphs just
as they can have for finite graphs.

\end{remark}

\section{Distinguishing Chromatic Number}
In this section we consider the distinguishing chromatic number for infinite graphs. We start with the following theorem which generalizes Theorem 4.5 of \cite{Collins} to infinite graphs.
\begin{theorem}
\label{DPC}
Let $G$ be a connected infinite graph with a finite maximum degree $\Delta$. Then $$\chi_{D}(G)\leq 2\Delta -1.$$
\end{theorem}
\begin{proof}
Since the only connected infinite graphs with $\Delta \leq 2$ are the ray and the double ray, and since the statement is true for both of them, we may assume that $G$ is a graph with $\Delta \geq 3$.

Let $v$ be a vertex of $G$ with degree $\Delta$ and $T$ be a BFS spanning tree of $G$ rooted at $v$. We colour the vertices of $G$ with numbers $\{1,\ldots,2\Delta-1\}$ as follows: Colour $v$ with the colour $2\Delta -1$ and colour its neighbours with $1,\ldots,\Delta$ according to their order in $T$. For the rest of the proof we try to keep $v$ as the only vertex of $G$ which has the property of being coloured with colour $2\Delta -1$ while all its neighbours are coloured by the colours $1,\ldots,\Delta$. We refer to this property as the ``property $*$" for future reference.

We proceed by colouring the vertices of $G$ in their order in $T$ such that each vertex is assigned the smallest colour that is not among its neighbours  in $G$ and siblings in $T$ that are already coloured (notice that there are at most $\Delta$ neighbours and at most $\Delta-2$ siblings). If this procedure ends with no vertex other than $v$ satisfying property $*$, then it is easy to see that the resulting colouring is a proper distinguishing one. However, if there are vertices other than $v$ with property $*$, we slightly modify the colouring.

So, suppose that a vertex other than $v$, denoted by $x_{2\Delta -1}$, has property $*$ and is the vertex with the smallest order in $T$ among such vertices. Then it has $\Delta$ neighbours $y_{1}, \ldots, y_{\Delta}$ which are respectively coloured by $1,\ldots,\Delta$. But this cannot happen unless $x_{2\Delta-1}$ has $\Delta-2$ siblings $x_{\Delta+1},\ldots,x_{2\Delta-2}$ with already assigned colours $\Delta+1,\ldots,2\Delta-2$, respectively (notice that $\Delta \geq 3$). Therefore, $x_{\Delta+1},\ldots,x_{2\Delta-2}$ are all of degree $\Delta$ and their neighbours create the palette $\{ 1, \ldots, \Delta \}$.

\begin{itemize}
\item[Case 1.] There is an $x_j \in \{ x_{\Delta+1},\ldots,x_{2\Delta-2} \}$ such that $N(x_{j}) \neq N(x_{2\Delta-1})$. Then, if we colour $x_{2\Delta-1}$ with the same colour as $x_{j}$, there is no colour preserving automorphism of $G$ (which also fixes $v$) mapping $x_{2\Delta-1}$ to $x_j$. Meanwhile, by this modification, the colouring of $G$ remains proper while $x_{2\Delta-1}$ has no longer property $*$.

\item[Case 2.] For all $x_j \in \{ x_{\Delta+1},\ldots,x_{2\Delta-2} \}$ we have $N(x_{j}) = N(x_{2\Delta-1})= \{ y_{1}, \ldots, y_{\Delta} \}$.
\begin{itemize}
\item[Subcase 2.1.] There is a vertex $y_s \in \{ y_{1}, \ldots, y_{\Delta}\}$ with no sibling or parent coloured with colour $2\Delta -1$. Then, colour $y_s$ with $2\Delta-1$ and colour $x_{2\Delta-1}$ with $s$. This modification keeps the colouring proper while neither $x_{2\Delta-1}$ nor $y_s$ has  property $*$.

\item[Subcase 2.2.] Each $y_s \in \{ y_{1}, \ldots, y_{\Delta}\}$ has a sibling or parent coloured with $2\Delta-1$. We should note that $y_s$ and its siblings have come before $x_{2\Delta -1}$ in the BFS ordering, so their parent cannot have property $*$ unless it is the root $v$. Meanwile, since $G$ is a connected infinite graph, it is impossible that all the vertices $y_1, \ldots, y_{\Delta}$ are children of $v$ since then $G$ must be isomorphic to $K_{\Delta,\Delta}$. Choose one $s$ from $\{ 1, \ldots , \Delta\}$ such that $y_s$ is not a child of $v$. We claim that there is an $i \in \{ 1,\ldots, \Delta \}$ which is neither a colour of a sibling of $y_s$ nor is the colour of its parent and $i\neq s$. This is so because $y_s$ has at most $\Delta-2$ siblings and one parent. Since the colour $2\Delta-1$ appeared at least once in this set, there must be an $i$ with the properties that we have claimed. Now colour $y_s$ by $i$ and colour $x_{2\Delta-1}$ with $s$. The colouring must remain proper, the vertex $v$ is the only predecessor of $x_{2\Delta-1}$ in the ordering of $T$ that satisfies property $*$, and all of them are fixed as long as $v$ is fixed.
\end{itemize}
\end{itemize}
By iteration of this procedure for all vertices with property $*$ other than $v$, we end up with a~proper distinguishing colouring of the vertices of $G$.
\end{proof}
\begin{remark}
The bound of Theorem \ref{DPC} is sharp since the double ray is a 2-regular graph with distinguishing chromatic number 3. Moreover, our proof can also be considered as another proof for Theorem 4.5 of \cite{Collins} in the following sense: in our proof, we used the fact that $G$ is an infinite graph in two places, one when we wanted to show that the statement is true for graphs with $\Delta=2$, and another when we explicitly exclude $G$ from being $K_{\Delta, \Delta}$ in Subcase 2.2. Treating the finite case, the graph $C_6$ has to be excluded because the only finite graphs with $\Delta=2$ and distinguishing chromatic number equal to 4 are $C_4$ and $C_6$. While $C_4$ is the graph $K_{2,2}$, $C_6$ is the only finite graph other than $K_{\Delta,\Delta}$ with the distinguishing chromatic number equal to $2\Delta$. Therefore, one can use our proof to show that all finite connected graphs other than $C_6$ and $K_{\Delta,\Delta}$ have distinguishing chromatic number at most $2\Delta-1$.
\end{remark}
\begin{example}
Now, an example of a graph with $\Delta \geq 3$ and $\chi_D (G) = 2\Delta-2$. Let $G$ be an infinite graph constructed in the following way. Take a complete bipartite graph $K_{\Delta, \Delta}$, delete an edge $uv$ of it and attach a copy of a ray to both vertices $u,v$.
Then the resulting graph $G$ is a connected infinite graph with $\Delta(G)=\Delta$ and $\chi_D=2\Delta-2$.
\end{example}
This example shows that the bound in the following conjecture cannot be lower than $2\Delta -2$. However, we do not know any infinite graph $G$ with $\chi_{D}(G)= 2\Delta -1$ yet. So, we have the following conjecture.
\begin{conjecture}
Let $G$ be a connected infinite graph with finite maximum degree $\Delta\geq 3$. Then $$\chi_{D}(G)\leq 2\Delta -2.$$
\end{conjecture}
 By the proof of Theorem \ref{DPC}, when we are colouring a vertex of a tree, there is no already coloured adjacent vertex other than its parent. Hence, $\Delta +1$ colours suffice to colour $T$ distinguishingly and thus we have the following corollary.
\begin{corollary}
\label{ITC}
Let $T$ be an infinite tree with finite maximum degree $\Delta$. Then $$\chi_D (T) \leq \Delta +1.$$
\end{corollary}
\begin{remark}
The bound of Corollary \ref{ITC} is sharp since the distinguishing chromatic number of the double ray is 3.
\end{remark}
For graphs with infinite motion the upper bound for distinguishing chromatic number can be lower, as it is shown in Theorem~\ref{tree-inf-motion} and Theorem~\ref{cubic}.
\begin{theorem}
\label{tree-inf-motion}
Let $T$ be an infinite, locally finite tree with infinite motion. Then $$\chi_{D}(T)\leq 3.$$
\end{theorem}
\begin{proof}
The statement is clear for trees with $\Delta \leq 2$, since the double ray is the only infinite tree with infinite motion and $\Delta = 2$, and it has distinguishing chromatic number 3. So, we assume that $T$ has a vertex $v$ of degree greater than or equal to 3. We view $v$ as the root of $T$. Then the $k$-th level $S_k$ of this rooted tree is the set of all vertices of distance $k$ from $v$. A ray $R=v_0v_1\ldots$ with $v_0\in S_k$ will be called  {\it simple} if  $v_i\in S_{k+i}$ for $i=1,2,\ldots$. By $C_k$ we denote the set of all vertices of $S_k$ that are origins of simple rays. Clearly, $C_k$ is non-empty since $T$ has infinite motion, and $C_k$ is finite since $T$ is locally finite.

Because $T$ is a tree, it has a proper vertex colouring with two colours, black and white. We acquire an extra colour, say red, and colour $v$ and all vertices of $S_2$ red. We will change this colouring into a distinguishing proper 3-colouring of $T$. To do this, we recursively define an infinite sequence $(k_n)$ of positive integers as follows: $k_1=3$ and $k_{n+1}=k_n+|C_{k_n}|+1,n\geq 1$. For every $n$ and for every vertex $u\in C_{k_n}$ we select a simple ray $R_u$ with origin $u$ and colour one vertex of $R_u$ red in such a way that there is exactly one red vertex in each level $S_k$ for $k=k_n+1,\ldots,k_n+|C_{k_n}|$.

Thus we obtain a proper 3-colouring of vertices of $T$ such that each level $S_k$, where $k\geq 3$, contains exactly one red vertex, unless $k=k_n$ for some $n$, and then there is no red vertex in such a level. Now, we prove that $c$ is distinguishing. First, due to our construction of $c$, it is easily seen that $v$ is the only vertex of $T$ whose degree is at least 3 and every vertex at distance 2 from it is red. Let $\varphi$ be a non-trivial automorphism preserving our colouring. Thus, $v$ is fixed by $\varphi$. Consequently, each level $S_k$ is mapped onto itself, hence each red vertex is fixed. Moreover, every simple ray is mapped by $\varphi$ onto a simple ray. Therefore, every vertex $u$ in $C_{k_n}$ is fixed by $\varphi$ because $u$ lies on the unique path between $v$ and another fixed red vertex and this path must be fixed point-wise.

 As the motion of $T$ is infinite, there exists a simple ray $R$ that is moved by $\varphi$. Clearly, $R$ shares a vertex with $C_{k_n}$ for infinitely many $n$'s. Each of them is fixed by any automorphism of $T$,  a contradiction.
\end{proof}

Using Theorem \ref{tree-inf-motion}, we can strengthen Corollary~\ref{ITC} for infinite trees distinct from the double ray.
\begin{theorem}
Let $T$ be an infinite tree with finite maximum degree $\Delta\geq 3$. Then $$\chi_D (T) \leq \Delta.$$
\end{theorem}
\begin{proof}
Let $R$ be a ray with an endvertex $v$ in $T$. We may assume that $v$ is not a pendant vertex of $T$. Consider $v$ as the root of the BFS ordering of vertices of $T$.  A finite subtree $T'$ is called {\it pendant} if it is a maximal finite subtree of $T$ such that all children of vertices of $T'$ also belong to $T'$, except for one vertex $v'$, called a {\it vertex of attachment} of $T'$, which is the first vertex of $T'$ in the BFS ordering. Let $T_1$ be a subtree of $T$ obtained by deleting all finite pendant subtrees of $T$, except their vertices of attachment. Thus $T_1$ has no pendant vertex, therefore it has infinite motion or it is the ray. By Theorem~\ref{tree-inf-motion}, the subtree $T_1$ admits  a distinguishing proper vertex $3$-colouring $c$.

Every automorphism $\varphi$ of $T$ maps a pendant vertex into a pendant one, therefore $\varphi$ maps $T_1$ onto itself. Thus each vertex of $T_1$ is fixed by every automorphism of $T$ preserving the colouring $c$, regardless of how we extend $c$. Consequently, every deleted pendant subtree $T'$ is mapped onto itself since its vertex of attachment is fixed.  Each vertex of $T'$ has at most $\Delta-1$ children. Step by step, according to the BFS ordering, we can properly extend $c$ to all vertices of $T'$ using $\Delta$ colours in such a way that siblings get distinct colours. If we do this for every deleted pendant subtree, we obtain a distinguishing proper colouring of $T$.
\end{proof}
To see that the inequality in Theorem 3.8 is tight, it suffices to consider an infinite graph obtained from a star $K_{1,\Delta}$ by substituting one edge by a ray.

\vspace{4mm}

Now, we return to graphs with infinite motion. The only graph with $\Delta (G)\leq 2$ and infinite motion is the double ray and its distinguishing chromatic number equals three.  The next theorem shows that the bound of Theorem~\ref{DPC} can be lowered for {subcubic} graphs (which are graphs with $\Delta\leq 3$) with infinite motion.

\begin{theorem}
\label{cubic}
Let $G$ be a connected subcubic graph with infinite motion. Then $$\chi_{D}(G)\leq 4.$$
\end{theorem}
\begin{proof}

To simplify our argument, we introduce some terms that are used in the proof. We say that some neighbours of a given vertex are coloured in a {\it standard way} if each of them obtains a distinct colour. A vertex is said to be \emph{fixed} if it is fixed by any colour-preserving automorphism of $G$. In the proof, we apply the following simple observation: if an already coloured vertex is fixed and some of its neighbours are coloured in a standard way, then these neighbours are also fixed.

At the beginning, we select a vertex $v_0$ that will be viewed as a "root" of $G$. Denote by $S_r$ a sphere in $v_0$ with a radius $r\geq 0$, that is $$S_r= S(v_0, r)=\{ u \in V(G) \;\vert \;\mathrm{d}(u,v_0)=r \}.$$

Let $y$ be a vertex of $G$. Every vertex $x\in N_G(y)$ such that $\mbox{d}(v_0,x)> \mbox{d}(v_0,y)$ is called an {\it up-neighbour}, and every vertex $x\in N_G(y)$ such that $\mbox{d}(v_0,x)< \mbox{d}(v_0,y)$ is called a {\it down-neighbour} of $y$.
Every sphere $S_r$ contains a nonempty subset $C_r$ of vertices that lie on simple rays originated at $v_0$.  Such a simple ray meets every sphere exactly once. Denote $D_r=S_r\setminus C_r$.

Every finite subgraph of an infinite subcubic graph $G$ can be properly coloured with three colours. Indeed, by Brooks' Theorem, the only subcubic finite graph $H$ with $\chi(H)>3$ is $K_4$, which cannot be a subgraph of an infinite connected subcubic graph.

\begin{figure}
\begin{center}
\includegraphics[scale=0.55]{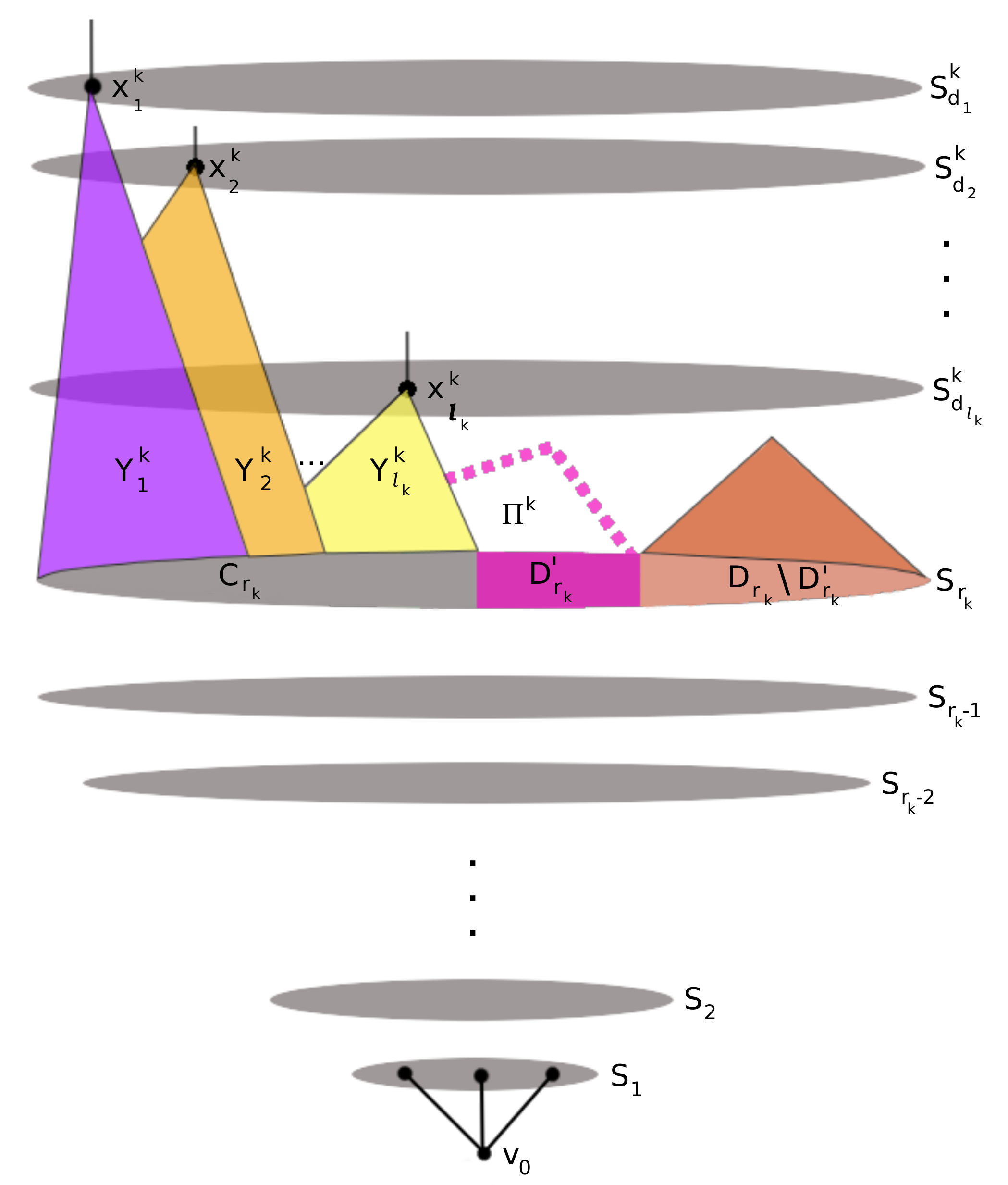}
\end{center}
\caption{Fixing vertices of a sphere $S_{r_k}$.}
\end{figure}

We colour the vertices of $G$ with four colours, where one of them, say black,  plays a special role.
First, we colour  the vertices of the union of four spheres $S_1\cup\ldots\cup S_4$ properly with three colours, not using black. Then we colour $v_0$ black, and in our colouring, we guarantee that $v_0$ is the only black vertex of $G$ whose distance from another black vertices is at least five. Thus $v_0$ will be fixed. Consequently, every sphere $S_r$ will be fixed set-wise. The main idea of the proof is to define a proper 4-colouring of $G$ which fixes infinitely many spheres $S_{r_k}, k\in \mathbb{N},$ point-wise.

We put the radius of the first sphere $r_1=7$. Let us point out that the spheres $S_{r_1-i}$ for $i=0,1,2$ are not coloured yet.
We select a set of vertices $X^1=\{x_1^1,\ldots, x_{l_1}^1\}$ such that every vertex of $C_{r_1}$ lie on a simple ray containing a vertex from $X$, and the distances $d_j^1=d(v_0,x_j^1)$ satisfy the inequalities
$$d_1^1> d_2^1> \ldots >d_{l_1}^1>r_1.$$
Moreover, we assume that the set $X^1$ is minimal. For each $j$, denote by $Y_j^1$ the set of all vertices which are on shortest paths from $x_j^1$ to $C_{r_1}$ (including $x_j^1$ and $C_{r_1}$). By minimality of $X^1$, each vertex of $X^1$ belongs to exactly one of the sets $Y^1_1,\ldots,Y^1_{l_1}$. Let $Y^1=\bigcup_{j=1}^{l_1}Y^1_j$. Clearly, each vertex of $Y^1$ has one or two up-neighbours and one or two down-neighbours. Observe that every down-neighbour of a vertex of $Y^1\setminus S_{r_1}$ also belongs to $Y^1$.

We colour $x_1^1$ black, and we colour properly the set $S_{d_1^1}\setminus\{x_1^1\}$ with three colours such that $x_1^1$ is the only black vertex in $S_{d_1^1}$, hence it is fixed. Now, starting with $y=x_1^1$ we fix all vertices of $Y_1^1$ by colouring them, sphere by sphere downwards, in the following way. Given an already coloured (and thus fixed) vertex $y\in Y_1^1$, we colour its down-neighbours in a standard way. Observe that this is always possible since if $y$ has two down-neighbours $y_1,y_2$, then each of them has at least one uncoloured down-neighbour, so we have at least two free colours for both $y_1,y_2$. Moreover, we assume that we use black colour only when it is necessary, i.e. when $y$ is not black and each of $y_1,y_2$ is adjacent to an already coloured vertex with the same colour distinct from black and from the colour of $y$. Thus we obtain a proper colouring of $Y_1^1\cup S_{d_1^1}$ fixing all vertices of $Y_1^1$.

Then for subsequent $j=2,\ldots,l_1$, we colour the set $Y^1_j$ as follows. Suppose we have already defined a proper 4-colouring of $$\bigcup_{i=1}^{j-1}(Y^1_i\cup S_{d_i^1})$$ fixing every vertex of $\bigcup_{i=1}^{j-1}Y^1_i$. We now want to colour the vertices of $Y^1_j\setminus \bigcup_{i=1}^{j-1}Y^1_i.$ By minimality of the set $X^1$, the vertex $x^1_j$ has a down-neighbour that does not belong to $\bigcup_{i=1}^{j-1}Y^1_i$, therefore $x^1_j$ has at most one already coloured neighbour $u$, and it is not black by our method of colouring. Indeed, $u$ had two uncoloured neighbours, when we were assigning a colour to it, namely its down-neighbour and $x^1_j$, so we had two non-black colours to use for $u$.

Hence, we can colour $x^1_j$ black. Then we properly colour yet uncoloured vertices of $S_{d_i^1}$, colouring a~vertex $v$ black only when we are forced to, that is, when $v$ has no up-neighbours and all its three down-neighbours are in $Y^1_j$ for $j<i$ and they are already coloured with three distinct colours different from black. Hence, $x_j^1$ is fixed as a unique black vertex in $S_{d_i^1}$ that was not already fixed, and that has an up-neighbour.  Analogously as in $Y^1_1$, we downwards colour properly the uncoloured vertices in $Y^1_j$, each time assigning a standard colouring to two uncoloured down-neighbours of an already coloured vertex, using black colour only if the other 3 colours are not possible.

Thus we obtain a proper colouring of $Y^1\cup \bigcup_{j=1}^{l_1}S_{d_j^1}$ fixing all vertices of $Y^1$, in particular those of $C_{r_1}$. Now we want to fix the vertices of $D_{r_1}$. Let $D'_{r_1}$ be the set of vertices of $D_{r_1}$ that belong to the connected components of   $G\left[\bigcup_{r\geq r_1}S_r\right]$ with  nonempty intersections with  $Y^1$.

For every vertex $u$ of $D'_{r_1}$, we choose a shortest path from $u$ to $Y^1$ omitting $S_{r_1-1}$. Let $\Pi^1$ be the set of vertices of all these paths. For $i\geq 0$,  let the {\it $i$-th layer} of $\Pi^1$ (with respect to $Y^1$) be the set $L_i$ of vertices of distance $i$ from $Y^1$. For each $i=0,1,\ldots$, we recursively colour in a standard way the uncoloured neighbours of each subsequent vertex of the $i$-th layer $L_i$, also those neighbours that do not belong to $\Pi^1$.  Thus, these neighbours will be fixed. Define
 $$ \pi^1=\max\{d(v_0,u)\,|\,u\in \Pi^1\}$$
 and observe that we can assume that $\pi^1<d^1_{l_1}-1$, hence we need not recolour any vertex. Indeed, if this was not the case, then each vertex $x^1_j$ could be substituted by a new one lying on a simple ray from $v_0$ through $x^1_j$, and the set $Y^1$ would be contained in a new one, but $\pi^1$ does not increase.  Consequently, every vertex of $D'_{r_1}$ is fixed. We then colour the vertices of $D_{r_1}\setminus D'_{r_1}$ properly with three colours and observe that they are fixed because every automorphism of $G$ moves infinitely many vertices. Hence, every vertex of $S_{r_1}$ is now fixed.

We then expand our colouring to a $4$-colouring of the subgraph induced by $\bigcup _{r=0}^{d_j^1}S_r.$ Observe that in general,
given a subcubic connected graph $G$ and a proper 4-colouring of its subgraph $H$ (not necessarily connected), we can expand this colouring to a proper 4-colouring of any larger subgraph of $G$. Indeed, taking a subsequent uncoloured vertex $v$, even if all three its neighbours are already coloured, we can use a fourth free colour. We define $r_2=r_1+d_1^1+3$.

 Next, for each subsequent $k\geq 2$, we fix every vertex of the sphere $S_{r_k}$ in the same way as we did it for $S_{r_1}$. In particular, we select a suitable minimal set $X^k=\{x^k_1,\ldots,x^k_{l_k}\}$, where the distances $d_j^k=d(v_0,x_j^k)$ satisfy the inequalities $d_1^k> d_2^k> \ldots >d_{l_k}^k>\pi^k+1>r_k.$
 Thus we obtain a~proper 4-colouring of the subgraph induced by $\bigcup _{r=0}^{d_j^k}S_r.$ Then we define the radius of the next sphere as $r_{k+1}=r_k+d^k_1+3$. Note that the spheres $S_{r_2-i}, i=0,1,2,$ are not coloured yet.

Finally, for every black vertex $y$, if there is no other black vertex within distance four from $y$, we pick a vertex $u$ such that $d(y,u)=2$ and $d(u,v_0)=d(y,v_0)-2$, and recolour it with black. Observe that afterwards the colouring is still proper, and all vertices $x^k_j$, $j=1,\ldots,l_k,$ and hence all sets $C_{r_k}$ remain fixed, because for each $k$, no vertex of the spheres $S_{r_k-1}, S_{r_k-2}$ was already coloured black.
 Moreover,  every vertex of $\Pi^k$  still has two neighbours $u_1,u_2$ in the next layer coloured with distinct colours because $d(u_1,u_2)\leq 2$, whence if $u_1$ had to be recoloured with black, then $u_1$ was of distance two from a certain black vertex $y$, and $u_2$ could not be black since $d(y,u_2)\leq 4$.  Consequently, every sphere $S_{r_k}$ is fixed.

We thus obtain a proper $4$-colouring $c$ of $G$ that fixes $v_0$ and every vertex of infinitely many spheres $S_{r_k}$, $k\geq1$. Suppose that a nontrivial automorphism $\varphi$ of $G$ preserves $c$. Then $\varphi$ moves a~certain vertex contained between two  spheres $S_{r_k}$ and $S_{r_{k+1}}$. Define $\psi(u)=\varphi(u)$ if $r_k<d(v_0,u)<r_{k+1}$, and $\psi(u)=u$ otherwise. Clearly, $\psi$ is an automorphism of $G$ that moves only finitely many vertices, which contradicts the assumption of infinite motion.
\end{proof}

We believe that the conclusion of the above theorem can be generalized to graphs with higher degrees as follows.
\begin{conjecture}
Let $G$ be a connected infinite graph with finite maximum degree $\Delta$ and infinite motion. Then $\chi_{D}(G)\leq \Delta +1$.
\end{conjecture}

\section{Distinguishing Chromatic Index and Total Distin- \\ guishing Chromatic Number}
In this section we consider infinite graphs and compare the distinguishing chromatic index with the chromatic index and the total distinguishing chromatic number with the total chromatic number. The following theorem gives a comparison for the latter one.
\begin{theorem}
\label{Tot}
Let $G$ be a connected infinite graph. Then $$\chi''_D (G) \leq \chi''(G)+1.$$
\end{theorem}
\begin{proof}
Whenever a graph is properly totally coloured and it has a vertex $v$ fixed by every automorphism, any neighbour of $v$ has to be mapped into itself because
every edge incident to $v$ has distinct colour. Hence, by induction on the distance from $v$, it can be concluded that such a total colouring is also distinguishing. So, if we colour one vertex of $G$ by an extra colour and pin it by any colour-preserving automorphism, then we can guarantee that it changes any proper total colouring of $G$ to be also distinguishing.
\end{proof}

The inequality in the above theorem is tight for the double ray $P_{\infty}$. Indeed, $\chi''(P_{\infty})=3$ and the proper total $3$-colouring is unique up to a permutation of colours. But this colouring is preserved by a shift of length three, so an extra colour is needed to break this automorphism.

It is more difficult to stabilize a vertex if we consider only edge colourings. We have the following result, which is true for finite graphs (cf. Theorem 16 of \cite{Kalinowski}).

\begin{theorem}
Let $G$ be a connected infinite graph. Then $$\chi'_D (G) \leq \chi'(G)+1.$$
\end{theorem}
\begin{proof} Similarly to the proof of Theorem \ref{Tot}, whenever a graph is properly edge coloured and it has a fixed vertex $v$, every neighbour of $v$ has to be mapped onto itself. Hence, by induction on the distance from $v$, it can be concluded that its edge colouring is also distinguishing. So, if we take a {\it geodesic} ray $R$ in $G$ (i.e., the distance in $G$ between any two vertices of $R$ equals their distance in $R$)  and colour three edges, namely the first, the third and the sixth one of this ray by an extra colour, then we have pinned the first vertex of it by any colour-preserving automorphism. Therefore, we can guarantee that it changes any proper edge colouring of $G$ to be also distinguishing.

Assume next that there exists no ray. Then the maximum degree of $G$, denoted by $\Delta$ is not finite, so $\chi'(G)$ is not finite, too. Then $\chi'_D (G) = \chi'(G)$ in this case. To see this, observe that the vertex set of $G$ is the union of finitely many spheres in some vertex $v$.
Since the number of vertices in every such sphere is at most $\Delta$ and the graph is connected, we have  $|V(G)|\leq\Delta$, hence also
$|E(G)|\leq\Delta$. Then it is possible to colour the edges of $G$ with $\Delta$ colours, using each
colour at most once. Such a colouring is proper and distinguishing, so $\chi'_D (G) = \chi'(G)$ if the maximum degree of $G$ is infinite.
\end{proof}
Due to a generalization of the renowned theorem of Vizing to infinite graphs (Theorem 7 of \cite{Behzad}), we know that $\Delta +2$ colours are enough to distinguishingly colour edges of a bounded-degree graph with no pair of incident edges with the same colour. However, we have the following conjecture which ends this section.
\begin{conjecture}
Let $G$ be a connected infinite graph with finite maximum degree $\Delta$. Then $$\chi'_D (G) \leq \Delta +1.$$
\end{conjecture}

\vspace{7mm}\noindent
{\bf Acknowledgment.} We are very indebted to an anonymous referee who encouraged us to clarify the proof of Theorem~\ref{cubic}. Thanks to that we improved our proof and obtained even a~stronger result than previously submitted.

\end{document}